\documentclass[final,reqno]{amsart}

\usepackage{graphicx}
\usepackage{amsmath,amsthm,amsfonts,amscd,amssymb}
\usepackage[color,notref,notcite]{showkeys}
\usepackage{url}
\usepackage{subeqnarray}
\usepackage[color,notref,notcite]{showkeys} 
\definecolor{labelkey}{rgb}{0.6,0,1}
\usepackage{enumitem}
\usepackage{algorithm}
\usepackage{algpseudocode}
\usepackage{citesort}

\theoremstyle{plain}
\newtheorem{theorem}{Theorem}[section]

\newtheorem{assumptions}[theorem]{Assumptions}

\theoremstyle{definition}
\newtheorem{definition}[theorem]{Definition}

\def\bhyp#1{\begin{equation}\label{#1}\begin{array}{c}}
\def\ehyp{\end{array}\end{equation}}

\newcounter{cst}

\theoremstyle{remark}

\numberwithin{equation}{section}
\numberwithin{figure}{section}

\newcommand{\RR}{{\mathbb R}}
\newcommand{\NN}{{\mathbb N}}

\def\O{\Omega}
\def\dsp{\displaystyle}

\def\disc{{\mathcal D}}
\def\mesh{{\mathcal M}}

\def\cF{{\mathcal F}}
\def\edge{\sigma}

\newcommand{\edgesext}{{{\cF}_{\rm ext}}} 
\newcommand{\edgesint}{{{\cF}_{\rm int}}} 

\def\dr{\partial}

\newcommand{\cI}{{\mathcal I}}
\DeclareMathOperator*{\argminB}{argmin}

\newcommand{\x}{\pmb{x}}

\newif\ifcorr\corrtrue
\usepackage[normalem]{ulem}
\definecolor{violet}{rgb}{0.580,0.,0.827}

\def\bpsi{{\boldsymbol \psi}}

\newcommand{\ud}{\, \mathrm{d}} 
\def\div{\mathop{\rm div}}

\title{The gradient discretisation method for the chemical reactions of biochemical systems}

\author{Yahya Alnashri}
\address[Yahya Alnashri]{Department of Mathematics, Al-Qunfudah University College, Umm Al-Qura University, Saudi Arabia}
\email{yanashri@uqu.edu.sa}

\author{Hasan Alzubaidi}
\address[Hasan Alzubaidi]{Department of Mathematics, Al-Qunfudah University College, Umm Al-Qura University, Saudi Arabia}
\email{hmzubaidi@uqu.edu.sa}



\date{\today}

\begin{document}
\newcommand{\subscript}[2]{$#1 _ #2$}

\begin{abstract}
We consider a biochemical model that consists of a system of partial differential equations based on reaction terms and subject to non--homogeneous Dirichlet boundary conditions. The model is discretised using the gradient discretisation method (GDM) which is a framework covering a large class of conforming and non conforming schemes. Under classical regularity assumptions on the exact solutions, the GDM enables us to establish the existence of the model solutions in a weak sense, and strong convergence for the approximate solution and its approximate gradient. Numerical test employing a finite volume method is presented to demonstrate the behaviour of the solutions to the model. 
\end{abstract}

\maketitle

\section{Introduction}
\label{introduction}
In this paper, we are interested in studding reaction diffusion equations:
\begin{subequations}\label{problem-rm}
\begin{align}
\partial_t \bar u(\x,t) -\mu_{1} \div(\nabla \bar u(\x,t))&=F(\bar u,\bar v), \quad (\x,t) \in \O\times (0,T),\label{rm1}\\
\partial_t \bar v(\x,t) -\mu_{2} \div(\nabla \bar v(\x,t))&=G(\bar u,\bar v), \quad (\x,t) \in \O\times (0,T\label{rm2}),
\end{align}
\end{subequations}
where $\mu_{1}$ and $\mu_{2}$ are the diffusion coefficients corresponding to the chemical concentrations $\bar u$ and $\bar v$ respectively. The functions $F$ and $G$ describe the governing kinetics of the chemical reactions and some biochemical phenomena have been expressed in literature based on the choice of these reaction terms. For an example, in 2014, Yung–Rong Lee and Sy-Sang Lia \cite{1} proposed a reaction-diffusion model that stimulated the relationship between concentrations of oxygen and lactic acid to simulate oxidative phosphorylation and glycolysis reactions in tissues. They concluded that a reaction-diffusion model can generate and maintain the ideal micro-environment for stem cells.

\par Another example is the Gray-Scott model, a very well-known reaction-diffusion system. The model describes the chemical reaction between two substances, an activator $\bar v$ and substrate $\bar u$ where both of which diffuse over time and so represents what is called cubic autocatalysis \cite{2}. 

 \par The Brusselator model is also an example of such chemical reaction systems which is used to describe mechanism of chemical reaction – diffusion with non-linear oscillations \cite{3}. The model occurs in a large number of physical problems such as the formation of ozone by atomic oxygen, in enzymatic reactions, and arises in laser and plasma physics from multiple coupling between modes\cite{4,5}. 

\par The gradient discretisation method (GDM) is a generic framework to design numerical schemes together with their convergence analysis for different models which are based on partial differential equations. It covers a variety of numerical schemes such as finite volumes, finite elements, discontinuous Galerkin, etc. We refer the reader to \cite{31,33,34,35,36,37,38,Y-2021} and the monograph \cite{30} for a complete presentation. The main purpose of this paper is the proof of convergence, in the GDM setting, of the approximate solutions to the weak solution of a system of reaction diffusion equations with non--homogeneous Dirichlet boundary conditions. The convergence is established without non--physical regularity assumptions on the solutions since it is based on discrete compactness techniques detailed in \cite{44}.

\par The paper is organised as follows. Section \ref{sec-conts} introduces the continuous model and its weak formulation. Section \ref{sec-disc} describes the gradient discretisation method for the model and states four required properties. Section \ref{sec-main} states the theorem corresponding to the convergence results. In Section \ref{sec-num}, we include numerical test by employing a finite volume scheme, namely the Hybrid Mimetic Mixed (HMM) method, to study and analyse the behaviour of the solutions of the Brusselator model as an example of the biochemical systems. The resultant relative errors with respect to the mesh size are also studied.

\section{Continuous model}\label{sec-conts}
We consider the following biochemical system of partial differential equations:
\begin{align}
\partial_t \bar u(\x,t) -\mu_1 \div(\nabla \bar u(\x,t))&=F(\bar u,\bar v), \quad (\x,t) \in \O\times (0,T),\label{rm-strong1}\\
\partial_t \bar v(\x,t) -\mu_{2} \div(\nabla \bar v(\x,t))&=G(\bar u,\bar v), \quad (\x,t) \in \O\times (0,T),\label{rm-strong2}\\
\bar u(\x,t)&= g, \quad (\x,t) \in\partial\O\times (0,T), \label{rm-strong3}\\
\bar u(\x,t)&= h, \quad (\x,t) \in\partial\O\times (0,T), \label{rm-strong6}\\
\bar u(\x,0)&=u_{\rm ini}(\x), \quad \x \in \O, \label{rm-strong4}\\
\bar v(\x,0)&=v_{\rm ini}(\x), \quad \x \in \O. \label{rm-strong5}
\end{align}

Our analysis focuses on the weak formulation of the above reaction diffusion model. Let us assume the following properties on the data of the model.
\begin{assumptions}\label{assump-rm}
The assumptions on the data in Problem \eqref{rm-strong1}--\eqref{rm-strong5} are the following:
\begin{itemize}
\item $\O$ is an open bounded connected subset of $\RR^d\; (d > 1)$, $T >0$, $\mu_1,\; \mu_2 \in \RR^+$,
\item $(u_{\rm ini},v_{\rm ini})$ are in $L^\infty(\O) \times L^\infty(\O)$,
\item $g$ and $h$ are traces of functions in $L^2(0,T;H^1(\O))$ whose time derivatives are in $L^2(0,T;H^{-1}(\O))$,
\item the functions $F,\; G: \RR^2 \to \RR$ are Lipschitz continuous with Lipschitz constants $L_F$ and $L_G$, respectively.
\end{itemize}
\end{assumptions}

Under Assumptions \ref{assump-rm}, the weak solution of \eqref{rm-strong1}--\eqref{rm-strong5} is seeking
\begin{subequations}\label{rm-weak}
\begin{equation*}
\begin{aligned}
&\bar u \in \{w\in L^2(0,T;H^1(\O))\;:\; \gamma(\bar u)=g \}, \bar v \in \{w\in L^2(0,T;H^1(\O))\;:\; \gamma(\bar v)=h \} \mbox{ such that }\\
&\bar u \in L^2(0,T;H^1(\O)) \cap C([0,T];L^2(\O)),\; \bar v\in L^2(0,T;H^1(\O)) \cap C([0,T];L^2(\O)),\\
&\forall \varphi \in L^2(0,T;H_0^1(\O)),\; \partial_t \varphi \in L^2(0,T;L^2(\O)),\; \varphi(\cdot,T)=0 ,\\
&\forall \psi \in L^2(0,T;H_0^1(\O)),\; \partial_t \psi \in L^2(0,T;L^2(\O)),\;\psi(\cdot,T)=0,
\end{aligned}
\end{equation*}
\begin{equation}\label{rm-weak1}
\begin{aligned}
-\dsp\int_0^T \dsp\int_\O \bar u(\x,t) \partial_t\varphi(\x,t) \ud \x  \ud t 
&+\mu_1\dsp\int_0^T\int_\O \nabla \bar u(\x,t) \cdot \nabla \varphi(\x,t)\ud \x \ud t,\\
&-\dsp\int_\O u_{\rm ini}(\x)\varphi(\x,0) \ud x\\
{}&\quad= \dsp\int_0^T\dsp\int_\O F(\bar u,\bar v)\varphi(\x,t) \ud \x \ud t,
\end{aligned}
\end{equation}
\begin{equation}\label{rm-weak2}
\begin{aligned}
-\dsp\int_0^T \dsp\int_\O  \bar v(\x,t)  \partial_t\psi(\x,t) \ud \x \ud t
&+\mu_2\dsp\int_0^T\int_\O \nabla \bar v(\x,t) \cdot \nabla \psi(\x,t)\ud \x \ud t\\ 
&-\dsp\int_\O v_{\rm ini}(\x)\psi(\x,0) \ud x\\
{}&\quad= \dsp\int_0^T\dsp\int_\O G(\bar u,\bar v)\psi(\x,t) \ud \x \ud t.
\end{aligned}
\end{equation}
\end{subequations}

\section{Discrete Problem}\label{sec-disc}
The analysis of numerical schemes for the approximation of solutions to our model is performed using the GDM. The first step to reach this analysis is the reconstruction of a set of discrete spaces and operators, which is called gradient discretisation.

\begin{definition}[Gradient discretisation for time--dependent problems with non--homogeneous Dirichlet boundary conditions]\label{def-gd-rm}
Let $\O$ be an open subset of $\RR^d$ (with $d > 1$) and $T>0$. A gradient discretisation for time--dependent problems with non--homogeneous Dirichlet boundary conditions is $\disc=(X_{\disc}, \mathcal I_{\disc,\dr}, \Pi_\disc, \nabla_\disc, J_\disc, (t^{(n)})_{n=0,...,N}) )$, where
\begin{itemize}
\item the set of discrete unknowns $X_{\disc}=X_{\disc,0}\oplus X_{\disc,\dr\O}$ is the direct sum of two finite dimensional spaces on $\RR$, corresponding respectively to the interior unknowns and to the boundary unknowns,
\item the linear mapping $\mathcal I_{\disc,\dr}: H^{\frac{1}{2}}(\dr\O) \to X_{\disc,\dr\O}$ is an interpolation operator for the trace,
\item the function reconstruction $\Pi_\disc : X_{\disc} \to L^2(\O)$ is a linear,
\item the gradient reconstruction $\nabla_\disc : X_{\disc} \to L^2(\O)^d$ is a linear and must be defined so that $\| \cdot ||_\disc := || \nabla_\disc \cdot ||_{L^2(\O)^d}$ defines a norm on $X_{\disc,0}$,
\item $J_\disc: L^\infty(\O) \to X_\disc$ is a linear and continuous interpolation operators for the
 initial conditions,
\item $t^{(0)}=0<t^{(1)}<....<t^{(N)}=T$ are discrete times. 
\end{itemize}
\end{definition}

Let us introduce some notations to define the space--time reconstructions $\Pi_\disc \varphi: \O\times[0,T]\to \RR$, and $\nabla_\disc \varphi: \O\times[0,T]\to \RR^d$, and the discrete time derivative $\delta_\disc \varphi : (0,T) \to L^2(\O)$, for $\varphi=(\varphi^{(n)})_{n=0,...,N} \in X_\disc^{N}$. 

For a.e $\x\in\O$, for all $n\in\{0,...,N-1 \}$ and for all $t\in (t^{(n)},t^{(n+1)}]$, let
\begin{equation*}
\begin{split}
&\Pi_\disc \varphi(\x,0)=\Pi_\disc \varphi^{(0)}(\x), \quad \Pi_\disc \varphi(\x,t)=\Pi_\disc \varphi^{(n+1)}(\x),\\
&\nabla_\disc \varphi(\x,t)=\nabla_\disc \varphi^{(n+1)}(\x).
\end{split}
\end{equation*}
Set $\delta t^{(n+\frac{1}{2})}=t^{(n+1)}-t^{(n)}$ and $\delta t_\disc=\max_{n=0,...,N-1}\delta t^{(n+\frac{1}{2})}$, to define
\begin{equation*}
\delta_\disc \varphi(t)=\delta_\disc^{(n+\frac{1}{2})}\varphi:=\frac{\Pi_\disc(\varphi^{(n+1)}-\varphi^{(n)})}{\delta t^{(n+\frac{1}{2})}}.\end{equation*}

\par Setting the gradient discretisation defined previously in the place of the continuous space and operators in the weak formulation of the model leads to a numerical scheme, called a gradient scheme.

\begin{definition}[Gradient scheme]\label{def-gs-rm} The gradient scheme for the continuous problem \eqref{rm-weak} is to find families of pair $(u^{(n)}, v^{(n)})_{n=0,...,N} \in  (\cI_{\disc,\dr}g+ X_{\disc,0}^{N+1}) \times (\cI_{\disc,\dr}h+ X_{\disc,0}^{N+1})$ such that $(u^{(0)}, v^{(0)})=(J_\disc u_{\rm ini}, J_\disc v_{\rm ini})$, and for all $n=0,...,N-1$, $u^{(n+1)}$ and $v^{(n+1)}$ satisfy the following equalities:
\begin{subequations}\label{rm-disc-pblm}
\begin{equation}\label{rm-disc-pblm1}
\begin{aligned}
&\dsp\int_\O \delta_\disc^{(n+\frac{1}{2})} u(\x) \Pi_\disc \varphi(\x)
+ \mu_1\dsp\int_\O \nabla_\disc u^{(n+1)}(\x) \cdot \nabla_\disc \varphi(\x)\ud \x\\
&\qquad=\dsp\int_\O F(\Pi_\disc u^{(n+1)}(\x), \Pi_{\disc} v^{(n+1)}(\x))\Pi_\disc \varphi(\x) \ud \x, \quad \forall \varphi \in X_{\disc,0}, \mbox{ and }
\end{aligned}
\end{equation}
\begin{equation}\label{rm-disc-pblm2}
\begin{aligned}
&\dsp\int_\O \delta_\disc^{(n+\frac{1}{2})} v(\x) \Pi_\disc \psi(\x)
+ \mu_2\dsp\int_\O \nabla_\disc v^{(n+1)}(\x) \cdot \nabla_\disc \psi(\x)\ud \x\\
&\qquad=\dsp\int_\O G(\Pi_\disc u^{(n+1)}(\x), \Pi_{\disc} v^{(n+1)}(\x))\Pi_\disc \psi(\x) \ud \x, \quad \forall \psi \in X_{\disc,0}.
\end{aligned}
\end{equation}
\end{subequations}
\end{definition}

In order to establish the stability and convergence of the above gradient scheme, sequences of gradient discretisations $(\disc_m)_{m\in\NN}$ described in Definition \ref{sec-disc} are required to satisfy four properties; coercivity, consistency, limit--conformity and compactness.

\begin{definition}[{Coercivity}]
Let $\disc$ be a gradient discretisation and let $C_\disc$ be defined by
\begin{equation}
C_\disc = \max_{\varphi\in X_{\disc,0}-\{0\}}\dsp\frac{\|\Pi_\disc \varphi \|_{L^2(\O)}}{\| \nabla_\disc \varphi \|_{L^2(\O)^d}}.
\end{equation}
A sequence $(\disc_m)_{m\in\NN}$ of a gradient discretisation is coercive if $C_\disc$ is bounded.
\end{definition}

\begin{definition}[Consistency]\label{def:cons-rm}
If $\disc$ is a gradient discretisation, define the function
$S_{\mathcal{D}} : H^1(\O)\to [0, +\infty)$ by, for $\varphi\in H^1(\O)$,
\begin{equation}
\begin{aligned}
S_{\mathcal{D}}(\varphi)= 
\min\Big\{\| \Pi_{\mathcal{D}} w - \varphi \|_{L^{2}(\Omega)}
+ \| \nabla_{\mathcal{D}} w - \nabla \varphi \|_{L^{2}(\Omega)^{d}}\;:\;
w\in X_\disc,\\ 
\mbox{ such that } w-\mathcal I_{\disc,\dr}\gamma\varphi \in X_{\disc,0}\Big\},
\end{aligned}
\end{equation}
A sequence $(\mathcal{D}_{m})_{m \in \mathbb{N}}$ of gradient discretisations is \emph{consistent} if, as $m \to \infty$ 
\begin{itemize}
\item for all $\varphi \in H^1(\O)$, $S_{\disc_m}(\varphi) \to 0$,
\item for all $w \in L^2(\O)$, $\Pi_{\disc_m}J_{\disc_m}w \to w$ in $L^2(\O)$,
\item $\delta t_{\disc_m} \to 0$.
\end{itemize}
\end{definition}

\begin{definition}[Limit-conformity]\label{def:lconf-rm}
If $\disc$ is a gradient discretisation, and the space $H_{\rm div}=\{\bpsi \in L^2(\O)^d\;:\; {\rm div}\bpsi \in L^2(\O) \}$, define the function $W_{\mathcal{D}} : H_{\rm div} \to [0, +\infty)$ by, for all $\bpsi \in H_{\rm div}$,
\begin{equation}\label{long-rm}
W_{\mathcal{D}}(\bpsi)
 = \sup_{w\in X_{\disc,0}\setminus \{0\}}\frac{\Big|\dsp\int_{\Omega}(\nabla_{\mathcal{D}}w\cdot \bpsi + \Pi_{\mathcal{D}}w \div (\bpsi)) \ud x \Big|}{\| \nabla_\disc w\|_{L^2(\O)^d} }.
\end{equation}
A sequence $(\disc_m)_{m\in \NN}$ of gradient discretisations is \emph{limit-conforming} if for all $\bpsi \in H_{\rm div}$, $W_{\disc_m}(\bpsi) \to 0$, as $m \to \infty$.
\end{definition}


\begin{definition}[Compactness]\label{def:compact}
A sequence of gradient discretisation $(\disc_m)_{m\in \NN}$ in the sense of Definition \ref{def-gd-rm} is \emph{compact} if for any sequence $(\varphi_m )_{m\in\NN} \in X_{\disc_m,0}$, such that $(\| \nabla_{\disc_m}\varphi_m \|_{L^2(\O)^d})_{m\in \NN}$ is bounded, the sequence $(\Pi_{\disc_m}\varphi_m )_{m\in \NN}$ is relatively compact in $L^2(\O)$.
\end{definition}

\begin{definition}[Dual norm on $\Pi_\disc(X_\disc)$]\label{def-dual}
If $\disc$ be a gradient discretisation, the dual norm $\| \cdot \|_{\star,\disc}$ on $\Pi_\disc(X_{\disc,0})$ is defined by, for all $w\in \Pi_\disc(X_{\disc,0})$,
\begin{equation}\label{eq-dual}
\| w \|_{\star,\disc}=\dsp\sup\Big\{ \dsp\int_\O w(\x)\Pi_\disc \varphi(\x)\ud \x \; :\; \varphi\in X_{\disc,0}, \| \nabla_\disc\varphi \|_{L^2(\O)^d} =1
\Big\}. 
\end{equation}
\end{definition}

\section{Convergence results}\label{sec-main}

Our convergence results are stated in the following theorem.

\begin{theorem}[Convergence of the gradient scheme]
\label{theorem-conver-rm} Assume \eqref{assump-rm} and let $(\disc_m)_{m\in\NN}$ be a sequence of gradient discretisations, that is coercive, consistent, limit-conforming and compact. For $m \in \NN$, let $(u_m,v_m) \in (\cI_{\disc_m,\dr}g+ X_{\disc_m,0}^{N+1}) \times (\cI_{\disc_m,\dr}h+ X_{\disc_m,0}^{N+1})$ be a solution to the gradient scheme \eqref{rm-disc-pblm} with $\disc=\disc_m$. Then there exists a weak solution $(\bar u,\bar v)$ of \eqref{rm-weak} and a subsequence of gradient discretisations, still denoted by $(\disc_m)_{m\in\NN}$, such that, as $m \to \infty$,
\begin{enumerate}
\item $\Pi_{\disc_m} u_m$ converges strongly to $\bar u$ in $L^\infty(0,T;L^2(\O))$,
\item $\Pi_{\disc_m} v_m$ converges strongly to $\bar v$ in $L^\infty(0,T;L^2(\O))$,
\item $\nabla_{\disc_m}u_m$ converges strongly to $\nabla\bar u$ in $L^2(\O\times(0,T))^d$,
\item $\nabla_{\disc_m}v_m$ converges strongly to $\nabla\bar v$ in $L^2(\O\times(0,T))^d$. 
\end{enumerate}
\end{theorem}

\begin{proof}
The proof relies on the compactness arguments as in \cite{30}, and is divided into four stages.\\

\noindent {\bf Step 1:}
Take liftings $\bar g \in L^2(0,T;H^1(\O))$ of $g$ and $\bar h \in L^2(0,T;H^1(\O))$ of $h$, such that $\gamma\bar g=g$ and $\gamma\bar h=h$. Thanks to the density of space--time tensorial functions in the space $L^2(0,T;H^1(\O))$ established in \cite[Corollary 1.3.1]{J-density}, we can express the liftings $\bar g$ and $\bar h$ in the following way: let $\ell \in \NN$, $(\phi)_{i=1,...,\ell},\; (\xi)_{i=1,...,\ell} \subset C^{\infty}([0,T])$, and $(p_i)_{i=1,...,\ell},\; (q_i)_{i=1,...,\ell} \subset H^1(\O)$ such that
\[
\bar g(\x,t)=\dsp\sum_{i=1}^{\ell}\phi_i(t)p_i(\x) \mbox{ and a.e. $\x\in\O$ and for all $t\in (0,T)$},
\]   
\[
\bar h(\x,t)=\dsp\sum_{i=1}^{\ell}\xi_i(t)q_i(\x) \mbox{ and a.e. $\x\in\O$ and for all $t\in (0,T)$}. 
\]   
Let $g_\disc,\; h_\disc \in X_{\disc}^{N_m+1}$ be defined by $g_\disc^{(n)}=\phi(t^{(n)})I_\disc p$ and $h_\disc^{(n)}=\xi(t^{(n)})I_\disc q$ for $0,...,N_m$, where
\[
I_\disc p = \argminB_{\varphi \in \cI_{\disc,\dr g}+X_{\disc,0}}\left(
\| \Pi_\disc \varphi-p \|_{L^2(\O)}
+\| \nabla_\disc \varphi-\nabla p \|_{L^2(\O)^d}
\right),
\]
\[
I_\disc q = \argminB_{\varphi \in \cI_{\disc,\dr h}+X_{\disc,0}}\left(
\| \Pi_\disc \varphi-q \|_{L^2(\O)}
+\| \nabla_\disc \varphi-\nabla q\|_{L^2(\O)^d}
\right).
\]
From the consistency property, as $m \to \infty$, we have
\begin{enumerate}
\item $\Pi_{\disc_m}g_{\disc_m} \to \bar g$ strongly in $L^2(\O \times (0,T))$ and $\Pi_{\disc_m}h_{\disc_m} \to \bar h$ strongly in $L^2(\O \times (0,T))$,
\item $\nabla_{\disc_m}g_{\disc_m} \to \nabla\bar g$ strongly in $L^2(\O \times (0,T))^d$ and $\nabla_{\disc_m}h_{\disc_m} \to \nabla\bar h$ strongly in $L^2(\O \times (0,T))^d$,
\item $\delta_{\disc_m}^{(n+\frac{1}{2})}g_{\disc_m} \to \dr_t \bar g$ strongly in $L^2(\O \times (0,T))$ and $\delta_{\disc_m}^{(n+\frac{1}{2})}h_{\disc_m} \to \dr_t \bar h$ strongly in $L^2(\O \times (0,T))$. 
\end{enumerate}
For any solution $(u,v)$ to the gradient scheme \eqref{rm-disc-pblm}, writing $w_1=u-g_\disc \in X_{\disc,0}$ and $w_2=v-h_\disc \in X_{\disc,0}$, we have for all $\varphi,\; \psi \in X_{\disc,0}$,
\begin{subequations}\label{rm-disc-pblm-new}
\begin{equation}\label{rm-disc-pblm1-new}
\begin{aligned}
&\dsp\int_\O \delta_\disc^{(n+\frac{1}{2})} w_1(\x) \Pi_\disc \varphi(\x)
+ \mu_1\dsp\int_\O \nabla_\disc w_1^{(n+1)}(\x) \cdot \nabla_\disc \varphi(\x)\ud \x
\\
&=\dsp\int_\O F(\Pi_\disc(w_1+g_\disc), \Pi_\disc(w_2+h_\disc))\Pi_\disc \varphi(\x) \ud \x\\
&\quad-\dsp\int_\O \delta_\disc^{(n+\frac{1}{2})} g_\disc(\x) \Pi_\disc \varphi(\x)
-\mu_1\dsp\int_\O \nabla_\disc g_\disc^{(n+1)}(\x) \cdot \nabla_\disc \varphi(\x)\ud \x,
\end{aligned}
\end{equation}
\begin{equation}\label{rm-disc-pblm2-new}
\begin{aligned}
&\dsp\int_\O \delta_\disc^{(n+\frac{1}{2})} w_2(\x) \Pi_\disc \psi(\x)
+ \mu_2\dsp\int_\O \nabla_\disc w_2^{(n+1)}(\x) \cdot \nabla_\disc \psi(\x)\ud \x
\\
&=\dsp\int_\O G(\Pi_\disc(w_1+g_\disc), \Pi_\disc(w_2+h_\disc))\Pi_\disc \psi(\x) \ud \x\\
&\quad-\dsp\int_\O \delta_\disc^{(n+\frac{1}{2})} h_\disc(\x) \Pi_\disc \psi(\x)
-\mu_2\dsp\int_\O \nabla_\disc h_\disc^{(n+1)}(\x) \cdot \nabla_\disc \psi(\x)\ud \x.
\end{aligned}
\end{equation}
\end{subequations}

\vskip 1pc
\noindent {\bf Step 2:} We need to have estimates on the quantities $\|\Pi_\disc w_1(t) \|_{L^\infty(0,T;L^2(\O))}$, $\|\Pi_\disc w_2(t) \|_{L^\infty(0,T;L^2(\O))}$,   $\|\nabla_\disc w_1 \|_{ L^2(\O \times (0,T))^d }$, and $\|\nabla_\disc w_2 \|_{ L^2(\O \times (0,T))^d }$.

Let $n \in \{ 0,...,N-1\}$ and put $\varphi:=\delta t^{ (n+\frac{1}{2}) }w_1^{(n+1)}$ in \eqref{rm-disc-pblm1-new} and $\psi:=\delta t^{ (n+\frac{1}{2}) }w_2^{(n+1)}$ in \eqref{rm-disc-pblm2-new}. We have
\begin{equation*}
\begin{aligned}
&\dsp\int_\O \Big(\Pi_\disc w_1^{(n+1)}(\x)-\Pi_\disc w_1^{(n)}(\x)\Big) \Pi_\disc w_1^{(n+1)}(\x) \ud \x\\
&\qquad+\mu_1\delta t^{(n+\frac{1}{2})}\int_\O\nabla_\disc w_1^{(n+1)}(\x) \cdot \nabla_\disc w_1^{(n+1)}(\x) \ud \x\\
&=\delta t^{(n+\frac{1}{2})}\dsp\int_\O F(\Pi_\disc (w_1^{(n+1)}(\x)+g_\disc), \Pi_\disc (w_2^{(n+1)}(\x)+h_\disc)) \Pi_\disc w_1^{(n+1)}(\x)  \ud \x\\
&\qquad-\delta t^{(n+\frac{1}{2})}\dsp\int_\O \delta_\disc^{(n+\frac{1}{2})} g_\disc(\x) \Pi_\disc w_1^{(n+1)}(\x) \ud \x\\
&\qquad-\mu_1\delta t^{(n+\frac{1}{2})}\int_\O\nabla_\disc g_\disc^{(n+1)}(\x) \cdot \nabla_\disc w_1^{(n+1)}(\x) \ud \x,\quad \mbox{and}
\end{aligned}
\end{equation*}
\begin{equation*}
\begin{aligned}
&\dsp\int_\O \Big(\Pi_\disc w_2^{(n+1)}(\x)-\Pi_\disc w_2^{(n)}(\x)\Big) \Pi_\disc w_2^{(n+1)}(\x) \ud \x\\
&\qquad+\mu_2\delta t^{(n+\frac{1}{2})}\int_\O\nabla_\disc w_2^{(n+1)}(\x) \cdot \nabla_\disc w_2^{(n+1)}(\x) \ud \x \ud t\\
&=\delta t^{(n+\frac{1}{2})}\dsp\int_\O G(\Pi_\disc (w_1^{(n+1)}(\x)+g_\disc), \Pi_\disc (w_2^{(n+1)}(\x)+h_\disc)) \Pi_\disc w_2^{(n+1)}(\x)  \ud \x\\
&\qquad-\delta t^{(n+\frac{1}{2})}\dsp\int_\O \delta_\disc^{(n+\frac{1}{2})} h_\disc(\x) \Pi_\disc w_2^{(n+1)}(\x) \ud \x\\
&\qquad-\mu_2\delta t^{(n+\frac{1}{2})}\int_\O\nabla_\disc h_\disc^{(n+1)}(\x) \cdot \nabla_\disc w_2^{(n+1)}(\x) \ud \x.
\end{aligned}
\end{equation*}
Apply the inequality, $a, b \in \RR$, $(a-b)a \geq \frac{1}{2}( |a|^2 -|b|^2 )$, to the first terms in the above equalities, sum on $n=0,...,m-1$, for some $m=0,...,N$ and apply the Cauchy--Schwarz inequality to obtain   
\begin{equation*}
\begin{aligned}
&\frac{1}{2}\| \Pi_\disc w_1^{(m)}\|_{L^2(\O)}^2-\frac{1}{2}\| \Pi_\disc w_1^{(0)} \|^2 
+\mu_1\sum_{n=0}^{m-1}\delta t^{(n+\frac{1}{2})} \|\nabla_\disc w_1^{(n)} \|_{L^2(\O)^d}^2 \\
&\leq \dsp\sum_{n=0}^{m-1}\delta t^{(n+\frac{1}{2})} \left\| F(\Pi_\disc (w_1^{(n+1)}+g_\disc^{(n+1)}), \Pi_\disc (w_2^{(n+1)}+h_\disc^{(n+1)})) \right\|_{L^2(\O)} \| \Pi_\disc w_1^{(n+1)} \|_{L^2(\O)}\\
&\quad+\sum_{n=0}^{m-1}\delta t^{(n+\frac{1}{2})}\|\Pi_\disc w_1^{(n+1)} \|_{L^2(\O)} \|\delta_\disc^{(n+\frac{1}{2})} g_\disc \|_{L^2(\O)}\\
&\quad+\mu_1\sum_{n=0}^{m-1}\delta t^{(n+\frac{1}{2})} \|\nabla_\disc w_1^{(n)} \|_{L^2(\O)^d} \|\nabla_\disc g_\disc^{(n)} \|_{L^2(\O)^d}, \quad \mbox{ and }
\end{aligned}
\end{equation*}
\begin{equation*}
\begin{aligned}
&\frac{1}{2}\| \Pi_\disc w_2^{(m)}\|_{L^2(\O)}^2-\frac{1}{2}\| \Pi_\disc w_2^{(0)} \|^2 
+\mu_1\sum_{n=0}^{m-1}\delta t^{(n+\frac{1}{2})} \|\nabla_\disc w_2^{(n)} \|_{L^2(\O)^d}^2 \\
&\leq \dsp\sum_{n=0}^{m-1}\delta t^{(n+\frac{1}{2})} \left\| G(\Pi_\disc (w_1^{(n+1)}+g_\disc^{(n+1)}), \Pi_\disc (w_2^{(n+1)}+h_\disc^{(n+1)})) \right\|_{L^2(\O)} \| \Pi_\disc w_1^{(n+1)} \|_{L^2(\O)}\\
&\qquad+\sum_{n=0}^{m-1}\|\Pi_\disc w_2^{(n+1)} \|_{L^2(\O)} \|\delta_\disc^{(n+\frac{1}{2})}h_\disc^{(n+1)} \|_{L^2(\O)}\\
&\qquad+\mu_2\sum_{n=0}^{m-1}\delta t^{(n+\frac{1}{2})} \|\nabla_\disc w_2^{(n)} \|_{L^2(\O)^d} \|\nabla_\disc h_\disc^{(n)} \|_{L^2(\O)^d}.
\end{aligned}
\end{equation*}
From the Lipschitz continuous assumptions on $F$ and $G$, one has, with letting $L:\max(L_F,L_G)$ and $C_0=\max(F({\bf 0}),G({\bf 0}))$,
\begin{equation*}
\begin{aligned}
&\frac{1}{2}\| \Pi_\disc w_1^{(m)}\|_{L^2(\O)}^2-\frac{1}{2}\| \Pi_\disc w_1^{(0)} \|^2 
+\mu_1\sum_{n=0}^{m-1}\delta t^{(n+\frac{1}{2})} \|\nabla_\disc w_1^{(n)} \|_{L^2(\O)^d}^2 \\
&\leq \dsp\sum_{n=0}^{m-1}\delta t^{(n+\frac{1}{2})}\Big[
L \Big( \| \Pi_\disc w_1^{(n)}\|_{L^2(\O)}^2 + \| \Pi_\disc w_1^{(n)}\|_{L^2(\O)} \| \Pi_\disc g_\disc^{(n)}\|_{L^2(\O)}\\
 &\qquad+ \| \Pi_\disc w_1^{(n)}\|_{L^2(\O)} \| \Pi_\disc w_2^{(n)}\|_{L^2(\O)}  + \| \Pi_\disc w_1^{(n)}\|_{L^2(\O)} \| \Pi_\disc h_\disc^{(n)}\|_{L^2(\O)}\Big)\\
&\qquad+C_0\| \Pi_\disc w_1^{(n)}\|_{L^2(\O)}
 \Big]
+\sum_{n=0}^{m-1}\delta t^{(n+\frac{1}{2})}\|\Pi_\disc w_1^{(n+1)} \|_{L^2(\O)} \|\delta_\disc^{(n+\frac{1}{2})}g_\disc \|_{L^2(\O)}\\
&\qquad+\mu_1\sum_{n=0}^{m-1}\delta t^{(n+\frac{1}{2})} \|\nabla_\disc w_1^{(n)} \|_{L^2(\O)^d} \|\nabla_\disc g_\disc^{(n)} \|_{L^2(\O)^d}, \quad \mbox{ and }
\end{aligned}
\end{equation*}
\begin{equation*}
\begin{aligned}
&\frac{1}{2}\| \Pi_\disc w_2^{(m)}\|_{L^2(\O)}^2-\frac{1}{2}\| \Pi_\disc w_2^{(0)} \|^2 
+\mu_1\sum_{n=0}^{m-1}\delta t^{(n+\frac{1}{2})} \|\nabla_\disc w_2^{(n)} \|_{L^2(\O)^d}^2 \\
&\leq \dsp\sum_{n=0}^{m-1}\delta t^{(n+\frac{1}{2})}\Big[
L \Big( \| \Pi_\disc w_2^{(n)}\|_{L^2(\O)}^2 + \| \Pi_\disc w_2^{(n)}\|_{L^2(\O)} \| \Pi_\disc g_\disc^{(n)}\|_{L^2(\O)}\\
 &\qquad+ \| \Pi_\disc w_2^{(n)}\|_{L^2(\O)} \| \Pi_\disc w_1^{(n)}\|_{L^2(\O)}  + \| \Pi_\disc w_2^{(n)}\|_{L^2(\O)} \| \Pi_\disc h_\disc^{(n)}\|_{L^2(\O)}\Big)\\
&\qquad+C_0\| \Pi_\disc w_2^{(n)}\|_{L^2(\O)}
 \Big]
+\sum_{n=0}^{m-1}\delta t^{(n+\frac{1}{2})}\|\Pi_\disc w_2^{(n+1)} \|_{L^2(\O)} \|\delta_\disc^{(n+\frac{1}{2})}h_\disc \|_{L^2(\O)}\\
&\qquad+\mu_1\sum_{n=0}^{m-1}\delta t^{(n+\frac{1}{2})} \|\nabla_\disc w_1^{(n)} \|_{L^2(\O)^d} \|\nabla_\disc h_\disc^{(n)} \|_{L^2(\O)^d}.
\end{aligned}
\end{equation*}
Then, using the Young's inequality in the right--hand side of the inequalities, we conclude
\begin{equation*}
\begin{aligned}
&\frac{1}{2}\| \Pi_\disc w_1^{(m)}\|_{L^2(\O)}^2-\frac{1}{2}\| \Pi_\disc w_1^{(0)} \|^2 
+\mu_1\sum_{n=0}^{m-1}\delta t^{(n+\frac{1}{2})} \|\nabla_\disc w_1^{(n)} \|_{L^2(\O)^d}^2 \\
&\leq M_1\dsp\sum_{n=0}^{m-1}\delta t^{(n+\frac{1}{2})}
\| \Pi_\disc w_1^{(n)}\|_{L^2(\O)}^2 
+\frac{1}{2\varepsilon_2}\dsp\sum_{n=0}^{m-1}\delta t^{(n+\frac{1}{2})}\| \Pi_\disc w_2^{(n)}\|_{L^2(\O)}^2\\ 
&\qquad+\frac{T}{2\varepsilon_4}C_0^2
+\frac{\mu_1}{2}\sum_{n=0}^{m-1}\delta t^{(n+\frac{1}{2})} \|\nabla_\disc w_1^{(n)} \|_{L^2(\O)^d}^2 
+TM_1C_1,\quad { and }
\end{aligned}
\end{equation*}
\begin{equation*}
\begin{aligned}
&\frac{1}{2}\| \Pi_\disc w_2^{(m)}\|_{L^2(\O)}^2-\frac{1}{2}\| \Pi_\disc w_2^{(0)} \|^2 
+\mu_2\sum_{n=0}^{m-1}\delta t^{(n+\frac{1}{2})} \|\nabla_\disc w_2^{(n)} \|_{L^2(\O)^d}^2 \\
&\leq M_3\dsp\sum_{n=0}^{m-1}\delta t^{(n+\frac{1}{2})}
\| \Pi_\disc w_2^{(n)}\|_{L^2(\O)}^2 
+\frac{1}{2\varepsilon_7}\dsp\sum_{n=0}^{m-1}\delta t^{(n+\frac{1}{2})}\| \Pi_\disc w_1^{(n)}\|_{L^2(\O)}^2\\ 
&\qquad+\frac{T}{2\varepsilon_9}C_0^2
+\frac{\mu_1}{2}\sum_{n=0}^{m-2}\delta t^{(n+\frac{1}{2})} \|\nabla_\disc w_2^{(n)} \|_{L^2(\O)^d}^2 
+TM_2C_1,
\end{aligned}
\end{equation*}
where $M_1:=L+\sum_{i=1}^5\dsp\frac{\varepsilon_i}{2}$, $M_2:=L+\sum_{i=6}^{10}\dsp\frac{\varepsilon_i}{2}$, and $C_1$ depends on $C_\disc$, $\|\nabla_\disc g_\disc\|_{L^2(\O)^d}$ and $\|\nabla_\disc h_\disc\|_{L^2(\O)^d}$, which are bounded. The desired estimates follow from combining the above inequalities together and take the supremum on $m=0,...,N$.

\vskip 1pc
\noindent {\bf Step 3:} We need to established estimates on $\| w_1 \|_{\star,\disc}$ and $\| w_2 \|_{\star,\disc}$. Take generic test functions $\varphi$ and $\psi$ in \eqref{rm-disc-pblm-new}. Use the Cauchy--Schwarz inequality to get, thanks to assumptions \eqref{assump-rm} and to the coercivity properties
\[
\begin{aligned}
&\dsp\int_\O \delta_\disc^{(n+\frac{1}{2})} w_1(\x) \Pi_\disc \varphi(\x) \ud \x\\
&\leq C_\disc \| \nabla_\disc\varphi \|_{L^2(\O)^d}
\Big[
L\Big( \| \Pi_\disc w_1^{(n+1)} \|_{L^2(\O\times(0,T))}
+\| \Pi_\disc w_2^{(n+1)} \|_{L^2(\O\times(0,T))}\\
\qquad&+\| \Pi_\disc g_\disc^{(n+1)} \|_{L^2(\O\times(0,T))}
+\| \Pi_\disc h_\disc^{(n+1)} \|_{L^2(\O\times(0,T))}\Big)+C_0\\
\qquad&
+\| \delta_\disc^{(n+\frac{1}{2})} g_\disc^{(n+1)} \|_{L^2(\O\times(0,T))}
+\mu_1 \| \nabla_\disc w_1^{(n+1)} \|_{L^2(\O\times(0,T))^d}
\Big],
\end{aligned}
\]
\[
\begin{aligned}
&\dsp\int_\O \delta_\disc^{(n+\frac{1}{2})} w_2(\x) \Pi_\disc \psi(\x) \ud \x\\
&\leq C_\disc\| \nabla_\disc\psi \|_{L^2(\O)^d}
\Big[
L\Big( \| \Pi_\disc w_1^{(n+1)} \|_{L^2(\O\times(0,T))}
+\| \Pi_\disc w_2^{(n+1)} \|_{L^2(\O\times(0,T))}\\
\qquad&+\| \Pi_\disc g_\disc^{(n+1)} \|_{L^2(\O\times(0,T))}
+\| \Pi_\disc h_\disc^{(n+1)} \|_{L^2(\O\times(0,T))}\Big)+C_0\\
\qquad&
+\| \delta_\disc^{(n+\frac{1}{2})} h_\disc^{(n+1)} \|_{L^2(\O\times(0,T))}
+\mu_1 \| \nabla_\disc w_2^{(n+1)} \|_{L^2(\O\times(0,T))^d}.
\Big].
\end{aligned}
\]
The desired estimates is then obtained by taking the supremum over $\varphi,\;\psi\in X_{\disc,0}$ with $\| \nabla_\disc \varphi  \|_{L^2(\O)^d}=\| \nabla_\disc \psi  \|_{L^2(\O)^d}=1$, multiplying by $\delta t^{(n+1)}$, summing over $n=0,...,N-1$, thanks to the estimates obtained in the previous step.

\vskip 1pc
\noindent {\bf Step 4:}
Owing to these estimates and the strong convergence of $\Pi_{\disc_m}g_{\disc_m}$, $\Pi_{\disc_m}h_{\disc_m}$, $\nabla_{\disc_m}g_{\disc_m}$, and $\nabla_{\disc_m}h_{\disc_m}$, the remaining of the proof is then similar to that of \cite[Theorem 3.2 ]{30}. 
\end{proof}
\section{Numerical Results}\label{sec-num}
To measure the efficiency of the gradient scheme \eqref{rm-disc-pblm} for the continuous problem \eqref{rm-strong1}--\eqref{rm-strong5}, we consider a particular choice of the gradient discrtisation method known as the Hybrid Mimetic Mixed (HMM) method, which is a kind of finite volume scheme and can be written in three different formats; the hybrid finite volume method \cite{D-2010-SUSHI}, the (mixed--hybrid) mimetic finite differences methods \cite{Brezzi-1991}, and the mixed finite volume methods \cite{T1997}. For the sake of completeness we briefly recall the definition of this gradient discretisation. Let $\mathcal T=(\mesh,\mathcal F,\mathcal P, \mathcal V)$ be the polytopal mesh of the spatial domain $\O$ used in the previous section and described in \cite[Definition 7.2]{30}. The elements of the gradient discretisation are: 
\begin{itemize}
\item The discrete spaces are 
\[
X_{\disc,0}=\{ v=((\varphi_{K})_{K\in \mathcal{M}}, (\varphi_{\sigma})_{\sigma \in \cF})\;:\; \varphi_{K},\, \varphi_{\sigma} \in \RR,\; \varphi_\edge=0,\; \forall \edge \in \cF \cap \dr\O
\},
\]
\[
\begin{aligned}
X_{\disc,\dr\O}=\{ &v=((v_{K})_{K\in \mathcal{M}}, (v_{\sigma})_{\sigma \in \cF})\;:\; v_{K} \in \RR,\, v_{\sigma} \in \RR,\\
&v_K=0 \mbox{ for all } K\in \mesh,\,
v_\sigma=0 \mbox{ for all } \sigma \in \edgesint\}.
\end{aligned}
\]
\item The non conforming a piecewise affine reconstruction $\Pi_\disc$ is defined by
\[
\begin{aligned}
&\forall \varphi\in X_{\disc}, \forall K\in \mesh, \mbox{ for a.e. } \x \in K,\\
&\Pi_\disc \varphi=\varphi_K\mbox{ on $K$}.
\end{aligned}
\]
\item The reconstructed gradients is piecewise constant on the cells (broken gradient), defined by
\[
\begin{aligned}
&\forall \varphi\in X_{\disc},\; \forall K\in\mathcal M,\,\forall \sigma\in\cF_K,\\
&\nabla_\disc \varphi=\nabla_{K}\varphi+
\frac{\sqrt{d}}{d_{K,\sigma}}R_K(\varphi)\mathbf{n}_{K,\sigma} \mbox{ on } D_{K,\edge},
\end{aligned}
\]
where a cell--wise constant gradient $\nabla_K \varphi$ and a stabilisation term $R_K(\varphi)$ are respectively defined by:
\[
\nabla_{K}\varphi= \dsp\frac{1}{|K|}\sum_{\sigma\in \cF_K}|\sigma|\varphi_\edge\mathbf{n}_{K,\sigma} \mbox{ and } R_K(\varphi)=(\varphi_\edge - \varphi_K - \nabla_K \varphi\cdot(x_\edge -x_K))_{\edge\in\cF_K},
\]
in which $x_\edge$ is centre of mass of $\edge$, $x_K$ is the gravity centre of cell $K$, $d_{K,\edge}$ is the orthogonal distance between $x_K$ and $\edge \in \cF_K$, ${\bf n}_{K,\edge}$ is the unit vector normal to $\edge$ outward to $K$ and $D_{K,\edge}$ is the convex hull of $\edge \cup \{x_K\}$.
\item The interpolant $J_\disc: L^\infty(\O) \to X_{\disc}$ is defined by:
\[
\begin{aligned}
&\forall w \in L^\infty(\O)\;:\; J_\disc w=((w_K)_{K\in \mesh},(w_\edge)_{\edge\in \cF}),\\
&\forall K\in\mesh,\; w_K=\dsp\frac{1}{|K|}\dsp\int_K w(\x) \ud \x \mbox{ and } \forall \edge\in\cF,\; w_\edge=0.
\end{aligned}
\]
\item the interpolant $\cI_{\disc,\dr}:H^{\frac{1}{2}}(\dr\O) \to X_{\disc,\dr\O}$ is defined by
\[
\begin{aligned}
\forall g\in H^{\frac{1}{2}}(\dr\O)\,:\,{}&(\cI_{\disc,\dr}g)_{\sigma}=\frac{1}{|\sigma|}\int_\sigma g(x)\ud s(x),\\
&\mbox{ for all } \sigma \in\edgesext \mbox{ such that } \sigma\subset\dr\O.
\end{aligned}
\]
\end{itemize}
The HMM scheme for \eqref{rm-disc-pblm} is the gradient scheme \eqref{rm-weak} written with the gradient discretisation constructed above.

\par As a test, we consider the Brusselator reaction-diffusion model \eqref{rm-strong1}--\eqref{rm-strong5} with non-homogeneous Dirichlet boundary conditions over the domain $\Omega=[0,1]^2$. 
The reaction functions in the Brusselator system are defined as
\begin{subequations}
\begin{align}
 F(\bar u,\bar v)&=a-(b +1)\bar u+\bar u^2\bar v, \nonumber \\
 G(\bar u,\bar v)&=b\bar u-\bar u^2\bar v,\nonumber
\end{align}
\end{subequations}
where $a$ is positive constant and $b$ is a parameter that can be varied to result in a range of different patterns. With $\x=(x,y)\in\O$, the exact solution in such a case is given as \cite{5}
\begin{subequations}\label{exact-rm}
\begin{align}
 \bar u(\x,t)&=\exp(-x-y-0.5t),\\
 \bar v(\x,t)&=\exp(x+y+0.5t),
\end{align}
\end{subequations}
with parameters chosen as $a=0$, $b=1$, $\mu_{1}=\mu_{2}=0.25$.
\par The initial and the Dirichlet boundary conditions are extracted from the analytical solutions \eqref{exact-rm}. The simulation is performed on a sequence of triangular meshes and is done up to $T=1$. The chosen meshes are of size $h=0.125$, $h=0.0625$, $h=0.03125$, and $h=0.015625$, respectively with time step is fixed as $0.0001$. Table 1 shows the relative errors on $\bar u$ and $\bar v$ and the corresponding rates of convergence with respect to the mesh size $h$. The resultant errors on the solutions $\bar u$ and $\bar v$ are proportional to the mesh size $h$, indicating that the HMM scheme behaves very well. 
\par Moreover, the $L_2$ relative errors on the gradients of the solutions with respect to the mesh size $h$ are shown in log-log scale Figure \ref{fig1}a for $\nabla\bar u$ and in Figure \ref{fig1}b for $\nabla\bar v$. A line of slope one is added in both figures as a reference. We observe that the relative errors on $\nabla\bar u$ and $\nabla\bar v$ scale linearly with $h$, giving a rate of convergence of one, which are compatible with behaviour expectations associated with the low-order methods such as the HMM method. 
{\color{red}
\begin{table}[]
\begin{tabular}{c c c c c}
\hline
$h$&
$\frac{\| \bar u(\cdot,t^{(n)}) - \Pi_\disc u^{(n)}\|_{L^{2}(\O)}}{\| \bar u(\cdot,t^{(n)})\|_{L^{2}(\O)}}$&
Rate&
$\frac{\| \bar v(\cdot,t^{(n)}) - \Pi_\disc v^{(n)}\|_{L^{2}(\O)}}{\| \bar v(\cdot,t^{(n)})\|_{L^{2}(\O)}}$&
Rate  
\\ \hline
0.125&
0.000720746&
--&
0.000561639&
--
\\ 
0.0625&
0.000184132&
1.968753&
0.000140295&
2.0011797
\\ 
0.03125&
0.0000501972&
1.8750586&
0.0000342813&
2.03296997
\\
0.015625&
0.0000149187&
1.750485&
0.00000688301&
2.31630842
\\ \hline
\end{tabular}
\caption{The relative errors and convergence rates w.r.t. the mesh size $h$ at time $t=1$ for the Brusselator model with parameters chosen as $a=0$, $b=1$, $\mu_{1}=\mu_{2}=0.25$.}
\label{tab-test-2}   
\end{table}
}
\begin{figure}[ht]
	\begin{center}
	\includegraphics[scale=0.9]{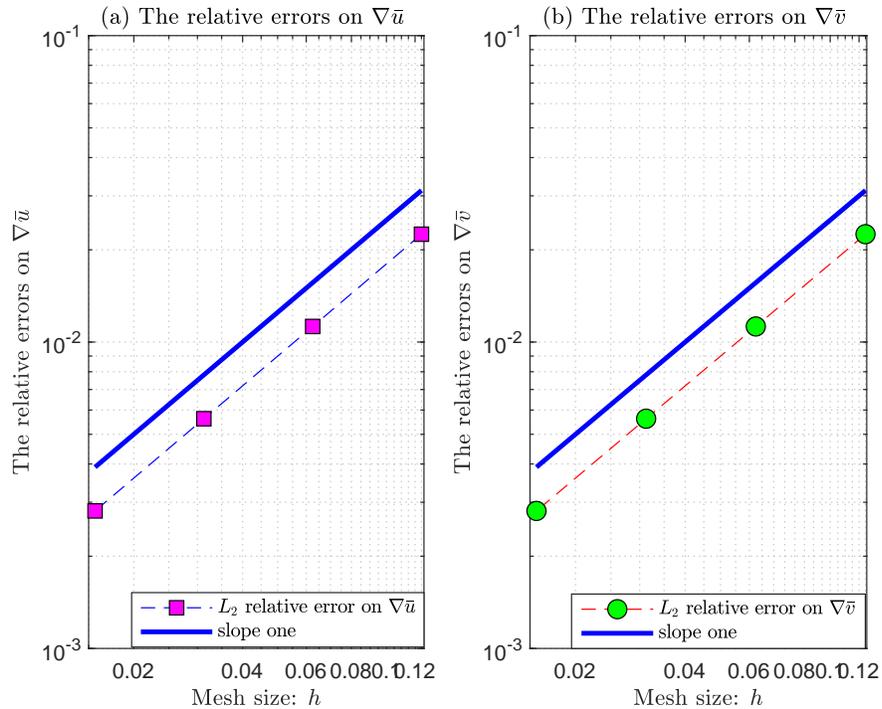}
	\end{center}
\caption{The relative errors on (a) $\nabla \bar u$ and (b) $\nabla \bar v$  w.r.t. the mesh size $h$ at time $t=1$ for the Brusselator model with parameters chosen as in Table 1.}
\label{fig1}
\end{figure}

\bigskip
\textbf{Acknowledgement}:\\
The authors would like to thank the Deanship of Scientific Research at Umm Al-Qura University for supporting this work Grant Code: 19-SCI-1-01-0027.


\bibliographystyle{siam}
\bibliography{BrusselatorRef}

\end{document}